\documentclass[12pt]{amsart}
\usepackage{amssymb,a4wide,xy}
\usepackage{amsmath,amscd,amsgen,amsmath,amstext,amsbsy,amsopn,amsfonts}
\xyoption{all}

\newtheorem{theorem}[subsubsection]{Theorem}
\newtheorem{proposition}[subsubsection]{Proposition}
\newtheorem{definition}[subsubsection]{Definition}
\newtheorem{corollary}[subsubsection]{Corollary}

\newtheorem{lemma}[subsubsection]{Lemma}

\def\al{\alpha}
\def\be{\beta}

\def\de{\delta}

\def\ep{\epsilon}

\def\sg{\sigma}
\def\pa{\partial}

\def\gm{\gamma}

\def\lra{\longrightarrow}
\def\rra{\rightarrow}

\def\lbr{\linebreak}

\def\otm{\otimes}
\def\ol{\overline}

\def\wt{\widetilde}

\def\uns{\underset}
\def\ovs{\overset}

\def\vthe{\vartheta}

\def\Im{\operatorname{Im}}
\def\Ker{\operatorname{Ker}}
\def\Coker{\operatorname{Coker}}
\def\Vec{\operatorname{\bf Vect}}
\def\Alg{\operatorname{\bf Alg}}
\def\XAlg{\operatorname{\bf {\mathcal X}Alg}}
\def\Tot{\operatorname{Tot}}
\def\XLie{\operatorname{{\bf {\mathcal{X}Lie}}}}
\def\Lie{\operatorname{{\bf Lie}}}
\def\XVec{\operatorname{{\bf {\mathcal{X}Vect}}}}

\begin{document}

\title{Cyclic Homologies of Crossed Modules of Algebras}
\author[G. Donadze]{Guram Donadze}
\address{\small \rm Guram Donadze: Department of Algebra\\A.Razmadze Mathematical Institute\\M.Alexidze St. 1, 0193 Tbilisi\\Georgia}
\email{donad@rmi.acnet.ge}
\author[N. Inassaridze]{Nick Inassaridze}
\address{\small \rm Nick Inassaridz: Department of Algebra\\A.Razmadze Mathematical Institute\\M.Alexidze St. 1, 0193 Tbilisi\\Georgia}
\curraddr{Departamento de Matem\'atica Aplicada I\\Universidad de
Vigo, EUIT Forestal\\36005 Pontevedra\\Spain}
\email{niko.inas@gmail.com}
\author[E. Khmaladze]{Emzar Khmaladze}
\address{\small \rm  Emzar Khmaladze: Department of Algebra\\A.Razmadze Mathematical Institute\\M.Alexidze St. 1, 0193 Tbilisi\\Georgia}
\email{e.khmal@gmail.com}
\author[M. Ladra]{Manuel Ladra}
\address{\small \rm  Manuel Ladra: Departamento de \'Algebra, Facultad de Matem\'aticas\\Universidad de Santiago de Compostela\\15782
Santiago de Compostela\\Spain} \email{manuel.ladra@usc.es}

\thanks{The first three authors would like to thank the {\em University of
Santiago de Compostela} for its hospitality during the work on this
paper. They were also partially supported by GNSF/ST06/3-004. The
authors were partially supported by INTAS grant 06-1000017-8609. The
second and the forth authors were supported by MEC grant MTM
2006-15338-C02 (European FEDER support included) and by Xunta de
Galicia grant PGIDITI06PXIB371128PR.}


\begin{abstract}
The Hochschild and (cotriple) cyclic homologies of crossed modules
of (not-necessarily-unital) associative algebras are investigated.
Wodzicki's excision theorem is extended for inclusion crossed
modules in the category of crossed modules of algebras. The cyclic
and cotriple cyclic homologies of crossed modules are compared in
terms of long exact homology sequence, generalising the relative
cyclic homology exact sequence.
\end{abstract}

\maketitle


\

\section{Introduction}\label{gen}

The present paper is concerned with the Hochschild and cyclic
homology theories of crossed modules of associative algebras, or
equivalently, of simplicial associative algebras with the associated
Moore complex of length $1$.

The study of (co)homological properties of similar objects in the
category of groups has been the subject of several papers, for
instance, the works of Baues \cite{Bau} and Ellis \cite{El5}
investigating the (co)homology of crossed modules of groups as the
(co)homology of their classifying spaces; the work of Carrasco,
Cegarra and Grandje\'an \cite{CCG} making the observation to the
same subject but in the cotriple (co)homology point of view; and
\cite{GLP} giving a connection between the cotriple homology of
crossed modules and the homology of their classifying spaces.

Crossed modules of groups were introduced by Whitehead in the late
1940s as algebraic models for path-connected CW-spaces whose
homotopy groups are trivial in dimensions $> 2$ \cite{Wh1}. Since
their introduction crossed modules have played an important role in
homotopy theory. For illustration we mention various classification
problems for low-dimensional homotopy types and derivation of van
Kampen theorem generalisations (see the survey of Brown \cite{Br}).

Crossed modules of Lie and associative algebras have also been
investigated by various authors. Namely, in the works of Dedecker
and Lue \cite{DL, L} crossed modules of associative algebras have
played a central role in what must be coefficients in
low-dimensional non-abelian cohomology. In \cite{BM} Baues and
Minian have shown that crossed modules of associative algebras can
be used to represent the Hochschild cohomology. In \cite{KL} Kassel
and Loday have used crossed modules of Lie algebras to characterize
the third Chevalley-Eilenberg cohomology of Lie algebras.

The aim of this paper is to construct and study the cotriple cyclic
homology of crossed modules of (non-unital) associative algebras,
generalising the classical cyclic homology of associative algebras
in zero characteristic case, and to compare it with the cyclic
homology of their nerves in terms of long exact homology sequence.

\subsection{Organisation}\label{gen}
After the introductory Section 1, the paper is organized in four
sections. Section 2 is devoted to recalling some necessary
definitions and results about crossed modules of associative
algebras and the Hochschild and cyclic homologies of simplicial
associative algebras. We begin Section 3 by examining the Hochschild
and cyclic homologies of aspherical augmented simplicial algebras
(Proposition \ref{prop_1} and Corollary \ref{cor_x}). Then we give
the five-term exact sequences relating the Hochschild and cyclic
homologies of crossed modules and algebras in low dimensions
(Theorem \ref{five}). Finally in this section, we investigate the
excision problem for Hochschild (resp. cyclic) homology of inclusion
crossed modules of algebras (Theorem \ref{exc1}). In Section 4 we
construct and study the cotriple cyclic homology theory in the
category of crossed modules of associative algebras. Then, we
calculate the cotriple cyclic homology of inclusion crossed modules
as the relative cyclic homology (Proposition \ref{relat}). In
Section 5 we compare the cyclic and cotriple cyclic homology
theories of crossed modules of associative algebras in terms of long
exact homology sequence (Theorem \ref{conection}).

\subsection{Notations and Conventions}\label{gen}
We fix $\mathbf k$ as a ground field. We make no assumptions on the
characteristic of $\mathbf k$, except as stated. All tensor products
are over $\mathbf k$. Moreover, $A^{\otm n}= A\otm \cdots \otm A$,
$n$ factors. Vectorspaces are considered over $\mathbf k$ and their
category is denoted by $\Vec$, while ${\mathfrak C}_{\ge 0}$ is the
category of non-negatively graded complexes of vectorspaces.
Algebras are (non-unital) associative algebras over $\mathbf k$ and
their category is denoted by $\Alg$. The term {\em free algebra}
means a free (non-unital) algebra over some vectorspace. Ideals are
always two-sided.

For any functor $T:{\underline C}\to \Vec$, where ${\underline C}$
coincides with the category $\Alg$ or the category of crossed
modules of algebras, and for any simplicial object $C_*$ in
${\underline C}$ denote by $T(C_*)$ the simplicial vectorspace
obtained by applying the functor $T$ dimension-wise to $C_*$.

\

\section{Preliminaries}

\subsection{Action of algebras and semidirect product}\label{subs_1}
Let $A$ and $R$ be two algebras. By an \emph{action} of $A$ on $R$
we mean an $A$-bimodule structure on $R$ satisfying the following
conditions:
$$
 a (r r')=(a r) r',\quad (r a) r'= r (a r'), \quad (r r') a = r (r' a)
$$
for all $a\in A$, $r, r'\in R$. For example, if $R$ is an ideal of
the algebra $A$, then the multiplication in $A$ yields an action of
$A$ on $R$.

Given an algebra action of $A$ on $R$, denote by $[A,R]$ the
sub-vectorspace of $R$ generated by the elements $[a,r]=ar - ra$ for
$r\in R$, $a\in A$. Moreover, one can form \emph{the semidirect
product} algebra, $R\rtimes A$, with the underlying vectorspace
$R\oplus A$ endowed with the multiplication given by
\begin{align*}
&(r, a)(r', a')=(r r'+ a r' + r a', a a')
\end{align*}
for $(r, a),(r', a')\in R\rtimes A$.

\subsection{Crossed module and its nerve}
Now we recall the basic notions about crossed modules of algebras
(cf. \cite{DIL} and \cite{El1}).

A \emph{crossed module} $(R, A,\rho)$ of algebras is an algebra
homomorphism $\rho:R\rra A$, together with an action of $A$ on $R$,
such that the following conditions hold:
\begin{align*}
&\rho(a r) = a \rho(r), \quad \rho(r a)=\rho(r) a,\\
&\rho(r) r' = r r' = r \rho(r') \qquad \qquad \qquad \qquad \qquad \qquad \qquad \qquad \text{(Peiffer identity)}
\end{align*}
for all $a\in A$, $r, r'\in R$. We point out that the image of
$\rho$ is necessarily an ideal of $A$, and that $\ker \rho$,
contained in the two-sided annihilator of $R$, is an $A/\rho
R$-bimodule.

The concept of a crossed module of algebras generalizes the concepts
of an ideal as well as a bimodule. In fact, a common instance of a
crossed module of  algebras is that of an algebra $A$ possessing an
ideal $I$; the inclusion homomorphism $I\hookrightarrow A$ is a
crossed module with $A$ acting on $I$ by the multiplication in $A$,
called \emph{the inclusion crossed module} of algebras.

Another common instance is that of an $A$-bimodule $M$ with trivial
multiplication; then the zero homomorphism $0:M\to A$, $m\mapsto 0$,
is a crossed module.

Any epimorphism of algebras $R\twoheadrightarrow A$ with the kernel
in the two-sided annihilator of $R$ is a crossed module, with $a\in
A$ acting on $r\in R$ by $a r=\tilde{r} r$ and $r a = r \tilde{r}$,
where $\tilde{r}$ is any element in the preimage of $a$.

A \emph{morphism} $(\mu,\nu):(R,A,\rho)\to(R',A',\rho')$ of crossed
modules is a commutative square of algebras
$$
\xymatrix{
R \ar[r]^{\mu}\ar[d]_\rho& R' \ar[d]^{\rho'}\\
A \ar[r]^{\nu}& A'}
$$
such that $\mu(a r)=\nu(a)\mu(r)$ and $\mu(r a)=\mu(r)\nu(a)$ for
$a\in A$, $r\in R$. Let us denote the category of crossed modules of
algebras by $\XAlg$.

Given a crossed module $(R,A,\rho)$ of algebras, consider the
semidirect product algebra, $R\rtimes A$. There are algebra
homomorphisms ${\mathfrak s}:R\rtimes A\to A$, $(r, a)\mapsto a$ and
${\mathfrak t}:R\rtimes A\to A$,  $(r, a)\mapsto \rho(r) + a$ and
binary operation $(r', a')\circ (r, a) = (r + r', a)$ for all pairs
$(r, a),(r', a')\in R\rtimes A$ such that $\rho(r) + a = a'$. This
composition $\circ$ with the \emph{source} map ${\mathfrak s}$ and
\emph{target} map ${\mathfrak t}$ constitutes an internal category
in the category $\Alg$ and the nerve of its category structure forms
the simplicial algebra $\mathfrak{N}_*(R,A,\rho)$, where
$\mathfrak{N}_n(R,A,\rho)=R\rtimes (\cdots (R\rtimes A)\cdots)$ with
$n$ semidirect factors of $R$, and face and degeneracy homomorphisms
are defined by
\begin{align*}
&d_0(r_1, \ldots, r_n, a) = (r_2, \ldots, r_n, a),\\
&d_i(r_1, \ldots, r_n, a) = (r_1, \ldots, r_i + r_{i+1}, \ldots, r_n, a),\qquad 0<i<n,\\
&d_n(r_1, \ldots, r_n, a) = (r_1, \ldots, r_{n-1}, \rho(r_n) + a),\\
&s_i(r_1, \ldots, r_n, a) = (r_1, \ldots, r_i, 0, r_{i+1}, \ldots, r_n, a),\qquad 0\le i\le n.
\end{align*}
The simplicial algebra $\mathfrak{N}_*(R,A,\rho)$ is called
\emph{the nerve of the crossed module} $(R,A,\rho)$.

\subsection{Homologies of simplicial algebras}\label{subsec_2.3}
Let us recall that for a given simplicial algebra $A_*$ its Moore
Normalisation is a complex of algebras ${\mathcal N} A_*$, where
$$
{\mathcal N}_nA_*=\ovs{n-1}{\uns{i=0}{\bigcap}} \Ker d^n_i\quad\text{and}\quad
\pa_n=d^n_n|_{\mathcal{N}_nA_*}.
$$
Note that the Moore complex of the nerve of a crossed module of
algebras $(R,A,\rho)$ is trivial in dimensions $\ge 2$ and is just
the original crossed module up to isomorphism with $R$ in dimension
$1$ and $A$ in dimension $0$.

The $n$-th homotopy of the simplicial algebra is defined as
$\pi_n(A_*)=\Ker\pa_n/\Im\pa_{n+1}$. Moreover, if it is given an
augmented simplicial algebra $(A_*,d^0_0, A)$, we calculate the
extended homotopy groups as $\pi_0(A_*,d^0_0, A)=\Ker
d^0_0/\Im\pa_1$ and $\pi_{-1}(A_*,d^0_0, A)=A/\Im d^0_0$. We say
that the augmented simplicial algebra $(A_*,d^0_0, A)$ is aspherical
if $\pi_n(A_*,d^0_0, A)=0$ for all $n\ge -1$. It is well-known that
in any homotopy group $\pi_n$, $n\ge 1$, the multiplication induced
by that of $A_n$ vanishes.

Given an algebra $A$, the standard bar, $C^{bar}(A)$, and
Hochschild, $C(A)$, complexes have the form
$$
C^{bar}_n(A)=C_n(A):=A^{\otm (n+1)},
$$
where the boundary operator of the bar complex is given by
$$
b'(a_0\otm \cdots \otm a_n)=\ovs{n-1}{\uns{i=0}{\sum}}(-1)^i(a_0\otm
\cdots \otm a_ia_{i+1}\otm \cdots \otm a_n),
$$
while the Hochschild boundary is given by
$$
b(a_0\otm \cdots \otm a_n)= b'(a_0\otm \cdots \otm a_n) + (-1)^n
(a_na_0\otm a_1\otm \cdots \otm a_{n-1}).
$$
Consider the \emph{cyclic}, first quadrant bicomplex:
\begin{equation}\label{cyc_bicom}
\CD
  & & @VbVV @V-b'VV @VbVV \\
  & & A^{\otm 2} @<{1-t}<< A^{\otm 2} @<N<< A^{\otm 2} @<{1-t}<< \\
  & & @VbVV @V-b'VV @VbVV \\
  & & A @<{1-t}<< A @<N<< A @<{1-t}<< ,
\endCD
\end{equation}
where $t:A^{\otm (n+1)}\to A^{\otm (n+1)}$, $n\ge 0$ is the cyclic
operator given by $t(a_0,\ldots,a_n)=(-1)^n(a_n,a_0,\ldots,a_{n-1})$
and $N:A^{\otm (n+1)}\to A^{\otm (n+1)}$ is the operator defined by
$N=1+t+t^2+ \cdots +t^n$. We denote by $CC(A)$ and $CC^{\{2\}}(A)$
the total complexes of the bicomplexes (\ref{cyc_bicom}) and that of
obtained through deleting all columns whose numbers are $\ge 2$ in
(\ref{cyc_bicom}), respectively.

Now suppose that we are given a functorial chain complex $\Phi(A)$,
as in the case of $C(A)$, $C^{bar}(A)$, $CC^{\{2\}}(A)$ and $CC(A)$
complexes, and set $H_n^{\Phi}(A)=H_n(\Phi(A))$, $n\ge 0$. Then,
extending these homology to simplicial algebras in a usual way, for
a given simplicial algebra $A_*$, denote by $\Phi(A_*)$ the
following bicomplex:
\begin{align*}
\CD
  & & @Vd_2VV @V-d_2VV @Vd_2VV \\
  & & \Phi_1(A_0) @<{}<< \Phi_1(A_1) @<<< \Phi_1(A_2) @<{}<< \\
  & & @Vd_1VV @V-d_1VV @Vd_1VV \\
  & & \Phi_0(A_0) @<{}<<\Phi_0(A_1) @<<< \Phi_0(A_2) @<{}<< ,
\endCD
\end{align*}
where horizontal differentials are obtained by taking alternating
sums, and by $H_n^{\Phi}(A_*)$ the $n$-th homology of its total
complex $\Tot(\Phi(A_*))$ (e.g. see \cite{GeWe}).

Given a crossed module of algebras $(R,A,\rho)$, denote by
$\Phi(R,A,\rho)$ the total complex
$\Tot(\Phi(\mathfrak{N}_*(R,A,\rho)))$. Then the \emph{Hochschild,
bar, naive Hochschild} and \emph{cyclic} homology of the crossed
module $(R,A,\rho)$ are defined by
\begin{align*}
&HH_n(R,A,\rho)=H_n(CC^{\{2\}}(R,A,\rho)),\qquad H^{bar}_n(R,A,\rho)=H_n(C^{bar}(R,A,\rho)),\\
&HH^{naive}_n(R,A,\rho)=H_n(C(R,A,\rho))\quad \text{and}\quad HC_n(R,A,\rho)=H_n(CC(R,A,\rho)),
\end{align*}
respectively, for $n\ge 0$.

\subsection{Linearly split extension}
Let $(R,A,\rho)$, $(S,B,\sigma)$ and $(T,C,\theta)$ be crossed
modules of algebras and
\begin{equation}\label{formula_2}
\xymatrix {
0 \ar[r] & R \ar[d]_{\rho}\ar^{\mu}[r] & S \ar^{\sigma}[d]\ar^{\mu'}[r] & T \ar^{\theta}[d]\ar[r] & 0\\
0 \ar[r] & A \ar^{\nu}[r] & B \ar^{\nu'}[r] & C \ar[r] & 0,
}
\end{equation}
a sequence of crossed modules. (\ref{formula_2}) is called an
extension of crossed modules if both of its rows are exact sequences
of algebras. An extension (2) of crossed modules of algebras is
called a linearly split extension if there exists a pair of linear
maps $\gm:T\to S$ and $\de:C\to B$ such that $\mu'\gm=1_T$,
$\nu'\de=1_C$ and $\sg\gm=\de\theta$.

\

\section{Hochschild and Cyclic Homology of Crossed Modules}\label{tres}

\subsection{Homologies of aspherical augmented simplicial algebras}
We begin this section with a few results about homology of
aspherical augmented simplicial algebras which we shall need in the
sequel.

\begin{proposition}\label{prop_1}
Let $(A_*,d^0_0, A)$ be an aspherical augmented simplicial algebra
and $\Phi:\Alg\to {\mathfrak C}_{\ge 0}$ be a covariant functor,
that for each $n\ge 0$, there is a functor $\wt\Phi_n:\Vec\to \Vec$
such that the diagram
\begin{equation*}
\xymatrix {
\Alg\ar[d]_{\mathcal U} \ar@<1mm>[rr]^{\Phi_n} & & \Vec \\
\Vec\ar@<1mm>[urr]^{\wt\Phi_n}
}
\end{equation*}
commutes, where $\mathcal U$ is the forgetful functor from the
category $\Alg$ to the category $\Vec$. Then
\begin{enumerate}
\item[(i)] the augmented simplicial vectorspace
$$
(\Phi_n(A_*),\Phi_n(d^0_0) ,\Phi_n(A))
$$
is acyclic for any $n\ge 0$;
\item[(ii)] there is a natural isomorphism
$$
H^{\Phi}_n(A) \cong H^{\Phi}_n(A_*),\quad n\ge 0\;.
$$
\end{enumerate}
\end{proposition}
\begin{proof} {\bf (i)} Straightforward from the fact that an acyclic augmented simplicial vectorspace
$({\mathcal U}(A_*), {\mathcal U}(d^0_0), {\mathcal U}(A))$ poses a
linear left (right) contraction.

\noindent{\bf (ii)} Let us consider the bicomplex $\Phi(A_*)$. Using (i), for any fixed $q$ the (horizontal) homology of the bicomplex
$\Phi(A_*)$ is $H_p(\Phi_q(A_*))=0$, $p> 0$ and $H_0(\Phi_q(A_*))=\Phi_q(A)$. Now the bicomplex spectral sequence argument completes the proof.
\end{proof}

\begin{corollary}\label{cor_x}
Let $(A_*,d^0_0,A)$ be an aspherical augmented simplicial algebra.
Then there are natural isomorphisms
\begin{align*}
&HH_n(A_*)\cong HH_n(A),\qquad \qquad \qquad HH^{bar}_n(A_*)\cong HH^{bar}_n(A),\\
&HH^{naive}_n(A_*)\cong HH^{naive}_n(A)\quad \text{and}\quad HC_n(A_*)\cong HC_n(A)
\end{align*}
for any $n\ge 0$.
\end{corollary}
\begin{proof}
It is clear that values of the functors $C_n$, $C^{bar}_n$,
$CC^{\{2\}}_n$ and $CC_n$, $n\ge 0$, on the algebra $A$ depend only
on the vectorspace underlying $A$. Thanks to Proposition
\ref{prop_1} (ii) the proof is completed.
\end{proof}

Note that if we are given an inclusion crossed module of algebras
$R\hookrightarrow A$, then $\mathfrak{N}_*(R\hookrightarrow A)$ is
an aspherical augmented simplicial algebra and by Corollary
\ref{cor_x} any its homology theory mentioned above coincides with
the respective homology of the quotient algebra $A/R$.

\subsection{Connes' Periodicity Exact Sequences}
The title of this subsection refers to the Connes' exact sequence
extended for crossed modules of algebras, connecting their
Hochschild and cyclic homologies in terms of long exact periodic
sequence. Namely, we have the following.

\begin{proposition}\label{prop_period}
Let $(R,A,\rho)$ be a crossed module of algebras. There is a natural
long exact sequence
\begin{align*}
\cdots \rra HH_n(R,A,\rho) \ovs{I}\rra HC_n(R,A,\rho) \ovs{S}\rra
HC_{n-2}(R,A,\rho) \ovs{B}\rra HH_{n-1}(R,A,\rho) \rra \cdots
\end{align*}
\end{proposition}
\begin{proof}
There is a short exact sequence of complexes
$$
0\lra CC^{\{2\}}(R,A,\rho) \lra CC(R,A,\rho)\lra CC(R,A,\rho)_{-2}\lra 0\;,
$$
where $CC_n(R,A,\rho)_{-2}=CC_{n-2}(R,A,\rho)$. This implies the result.
\end{proof}

\subsection{Five-term exact sequences}
Now we establish the five-term exact sequences relating the
Hochschild and cyclic homologies of crossed modules of algebras and
their cokernel algebras.

\begin{theorem}\label{five}
Let $(R,A,\rho)$ be a crossed module of algebras. There are exact
sequences of vectorspaces
\begin{align*}
& HH_2(R,A,\rho)\lra HH_2(\Coker\rho)\lra \Ker\rho/[A,\Ker\rho]\lra HH_1(R,A,\rho)\\
& \lra HH_1(\Coker\rho)\lra 0,\\
& HC_2(R,A,\rho)\lra HC_2(\Coker\rho)\lra \Ker\rho/[A,\Ker\rho]\lra HC_1(R,A,\rho)\\
& \lra HC_1(\Coker\rho)\lra 0
\end{align*}
and equality
$$
HH_0(R,A,\rho)=HC_0(R,A,\rho)=\Coker\rho\big/[\Coker\rho, \Coker\rho].
$$
\end{theorem}
\begin{proof}
We will only prove the exactness of the first sequence. The proof
for the second is essentially the same and left to the reader.

Consider the bicomplex $CC^{\{2\}}(\mathfrak{N}_*(R,A,\rho))$. Then
there is a first quadrant spectral sequence
$$
E^1_{pq}=H_q(CC^{\{2\}}_p(\mathfrak{N}_*(R, A, \rho)))\Rightarrow
HH_{p+q}(R,A,\rho).
$$
It is easy to check that we have
$$
E^1_{p0}=CC^{\{2\}}_p(\Coker\rho),\quad p\ge 0\quad\text{and}\quad
E^1_{01}=\pi_1(\mathfrak{N}_*(R,A,\rho))=\Ker \rho.
$$
Using the Eilenberg-Zilber Theorem and the K\"unneth Formula we also
have
\begin{align*}
&E^1_{11}=\Big(\pi_1(\mathfrak{N}_*(R,A,\rho))\otimes \pi_0(\mathfrak{N}_*(R,A,\rho))\Big) \oplus \Big(\pi_0(\mathfrak{N}_*(R,A,\rho))\otimes \pi_1(\mathfrak{N}_*(R,A,\rho))\Big)\\
&\oplus \pi_1(\mathfrak{N}_*(R,A,\rho)) =\Big(\Ker\rho \otimes \Coker\rho \Big)\oplus \Big(\Coker\rho \otimes \Ker\rho \Big)\oplus \Ker\rho.
\end{align*}
Continuing calculations we deduce that
$$
E^{\infty}_{00}=E^2_{00}=HH_0(\Coker\rho) \quad \text{and} \quad
E^{\infty}_{10}=E^2_{10}=HH_1(\Coker\rho).
$$
Moreover,
$$
E^2_{01}=\Coker\{E^1_{11}\rightarrow E^1_{01}\}=\Ker\rho /[\Ker\rho
, \Coker\rho]=\Ker\rho /[A,\Ker\rho].
$$
Therefore, we have a differential $d^2:HH_2(\Coker\rho)\rightarrow
\Ker\rho /[A,\Ker\rho]$ of the spectral sequence, which determines
the base term $E^\infty_{20}$ and the fiber term $E^\infty_{01}$
from the following exact sequence:
\begin{align}\label{equation_1}
0\rightarrow E^{\infty}_{20}\rightarrow HH_2(\Coker\rho)\rightarrow
\Ker\rho /[A,\Ker\rho]\rightarrow E^{\infty}_{01} \rightarrow 0\;.
\end{align}
Clearly, we have the short exact sequence
\begin{align}\label{equation_2}
0\rra E^{\infty}_{01}\rra HH_1(R,A,\rho)\rra E^{\infty}_{10}\rra 0
\end{align}
and the epimorphism
\begin{align}\label{equation_3}
HH_2(R,A,\rho)\twoheadrightarrow E^{\infty}_{20}.
\end{align}
Now (\ref{equation_1}), (\ref{equation_2}) and (\ref{equation_3}) imply the required result.
\end{proof}

This statement shows that the Hochschild (resp. cyclic) homology of
a crossed module $(R,A,\rho)$ differs, in general, from the
Hochschild (resp. cyclic) homology of the cokernel algebra $\Coker
\rho$.

\subsection{Excision property}
In this subsection we discus some aspects of the excision property
for Hochschild (resp. cyclic) homology of crossed modules. Namely,
we give sufficient conditions for inclusion crossed modules of
algebras when they satisfy the excision property. The concept of
investigation of the problem in general setting according to
Wodzicki is substantially different and will be treated in a
separate paper.

The excision problem for Hochschild (resp. cyclic) homology in the
category of crossed module of algebras is formulated as follows: let
\begin{equation}\label{mimdevr_0}
0 \lra (R,A,\rho) \xrightarrow{(\mu,\nu)} (S,B,\sigma) \xrightarrow{(\eta,\theta)} (T,C,\tau) \lra 0
\end{equation}
be a linearly split extension of crossed modules of algebras. The
crossed module $(R,A,\rho)$ is excisive (or satisfies excision) for
Hochschild (resp. cyclic) homology in the category of crossed
modules if the induced natural long homology sequence
$$
\cdots \rra HH_n(R,A,\rho)\rra HH_n(S,B,\sigma)\rra HH_n(T,C,\tau) \rra HH_{n-1}(R,A,\rho)\rra \cdots
$$
(resp.
$$
\cdots\rra HC_n(R,A,\rho)\rra HC_n(S,B,\sigma)\rra HC_n(T,C,\tau) \rra HC_{n-1}(R,A,\rho)\rra \cdots)
$$
is exact for any linearly split extension (\ref{mimdevr_0}) of the
crossed module $(R,A,\rho)$.

Note that according to the Connes' Periodicity Exact Sequence,
Proposition \ref{prop_period}, the excision properties for
Hochschild and cyclic homologies in the category of crossed modules
of algebras are equivalent.

The aim of this subsection is to prove the following.

\begin{theorem}\label{exc1}
Let $(I,A,inc)$ be an inclusion crossed module of algebras such that
$H_n^{bar}(I,A,inc)=0$, $n\ge 0$. Then $(I,A,inc)$ is excisive for
Hochschild homology.
\end{theorem}
\begin{proof}
Consider any linearly split extension (\ref{mimdevr_0}) of crossed
modules of algebras. Hence we have the short exact sequence
\begin{equation}\label{formula_4}
\CD 0 \to \pi_0(\mathfrak{N}_*(I,A,inc))
@>{\pi_0(\mathfrak{N}_*(\mu,\nu))}>>
\pi_0(\mathfrak{N}_*(S,B,\sigma))\\
@>{\pi_0(\mathfrak{N}_*(\eta,\theta))}>>
\pi_0(\mathfrak{N}_*(T,C,\tau)) \to 0.
\endCD
\end{equation}
and the isomorphism
\begin{equation}\label{formula_5}
\CD
\pi_1(\mathfrak{N}_*(S,B,\sigma)) @>{\ovs{\pi_1(\mathfrak{N}_*(\eta,\theta))}\cong}>> \pi_1(\mathfrak{N}_*(T,C,\tau)).
\endCD
\end{equation}
It is also easy to see that we have the following commutative
diagram with exact rows of complexes:
$$
\xymatrix {
0\ar[r] & C(I,A,inc)\; \ar[d]_{C(\mu,\nu)}\ar[r] & CC^{\{2\}}(I,A,inc) \ar[d]_{CC^{\{2\}}(\mu,\nu)}\ar[r] & C^{bar}(I,A,inc)_{-1} \ar[d]^{C^{bar}(\mu,\nu)} \ar[r] & 0\;\\
0\ar[r] & \Ker\big(C(\eta,\theta)\big)\; \ar[r] & \Ker\big(CC^{\{2\}}(\eta,\theta)\big) \ar[r] & \Ker\big(C^{bar}(\eta,\theta)\big)_{-1} \ar[r] & 0,
}
$$
where $C^{bar}_n(I,A,inc)_{-1}=C^{bar}_{n-1}(I,A,inc)$ and $\Ker_n
\{C^{bar}(\eta,\theta)\}_{-1}=\Ker_{n-1} \{C^{bar}(\eta,\theta)\}$.
Clearly, the induced commutative diagram of long exact homology
sequences and the five lemma implies that, if $C(\mu,\nu)$ and
$C^{bar}(\mu,\nu)$ are quasi-isomorphisms, then so is
$CC^{\{2\}}(\mu,\nu)$.

We shall only prove that $C(\mu,\nu)$ is quasi-isomorphic. The proof
that $C^{bar}(\mu,\nu)$ is quasi-isomorphic as well can be
accomplished essentially in the same way and will be omitted.

We need the following.

\begin{lemma}\label{lemexc1}
Let $(S,B,\sigma)$ be any crossed module of algebras, then there is
an isomorphism
\begin{multline*}
H_q(C_p(\mathfrak{N}_*(S,B,\sigma)))\cong \\ \left\lbrace
\begin{matrix}
\uns{i_0 + \cdots + i_p = q }{\bigoplus}\big( \pi_{i_0}(\mathfrak{N}_*(S,B,\sigma))\otimes \cdots \otimes \pi_{i_p}(\mathfrak{N}_*(S,B,\sigma)) \big)& \text{for} & p\ge 0, \; 0\le q \le p+1\\
 \qquad 0 \hfill & \text{for} & 0\le p < q-1
\end{matrix}
\right.
\end{multline*}
with $i_0, \ldots, i_p=0\;\text{or}\;1$. Moreover, through this
isomorphism the Hochschild differential behaves as follows:
$$
x_0\otimes x_1\otimes \cdots \otimes x_p \mapsto
\ovs{p-1}{\uns{j=0}{\sum}}(-1)^j (x_0\otimes \cdots \otimes x_j
x_{j+1}\otimes \cdots \otimes x_p) + (-1)^p\; (x_p x_0 \otimes
\cdots \otimes x_{p-1}),
$$
where $x_j\in \pi_{i_j} (\mathfrak{N}_*(S,B,\sigma))$, $0\le j\le
p$, and the multiplication $x_j x_{j'}$ is meant as the
multiplication in $\pi_0 (\mathfrak{N}_*(S,B,\sigma))$ or the
bimodule structure of $\pi_0 (\mathfrak{N}_*(S,B,\sigma))$ on $\pi_1
(\mathfrak{N}_*(S,B,\sigma))$.
\end{lemma}
\begin{proof}
The proof, requiring to use again the Eilenberg-Zilber Theorem and
the K\"unneth Formula, is routine and will be omitted.
\end{proof}

\;

Returning to the main proof, consider the bicomplexes
$C(\mathfrak{N}_*(I,A,inc))$ and \lbr $ \CD {\mathbb M}\equiv
\Ker\Big(C(\mathfrak{N}_*(S,B,\sigma))@>{C(\mathfrak{N}_*(\eta,\theta))}>>
C(\mathfrak{N}_*(T,C,\tau))\Big).
\endCD
$
Then there are first quadrant spectral sequences
$$
E^1_{pq}=H_q(C_p(\mathfrak{N}_*(I,A,inc)))\Rightarrow
HH^{naive}_{p+q}(I,A,inc)
$$
and
$$
\CD
\ol{E}^1_{pq}=H_q\Big(\Ker\big(C_p(\mathfrak{N}_*(S,B,\sigma))@>{C_p(\mathfrak{N}_*(\eta,\theta))}>> C_p(\mathfrak{N}_*(T,C,\tau))\big)\Big)\Rightarrow H_{p+q}\big(\Tot({\mathbb M})\big).
\endCD
$$
Moreover, we have the linearly split exact sequence of simplicial
vectorspaces
$$
0\rra {\mathbb M}_{*p} \rra C_p(\mathfrak{N}_*(S,B,\sigma)) \rra
C_p(\mathfrak{N}_*(T,C,\tau)) \rra 0,
$$
implying the short exact homology sequence
$$
0\rightarrow \ol{E}^1_{pq} \rra H_q(C_p(\mathfrak{N}_*(S,B,\sigma)))
\rra H_q(C_p(\mathfrak{N}_*(T,C,\tau))) \rra 0.
$$
Hence, from Lemma \ref{lemexc1} we have an isomorphism
\begin{multline}\label{formula_3}
\ol{E}^1_{pq}\cong \left\lbrace
\begin{matrix}
\uns{i_0 + \cdots + i_p = q }{\bigoplus}\Ker \big\{ \pi_{i_0}(\mathfrak{N}_*(\eta,\theta))\otimes \cdots \otimes \pi_{i_p}(\mathfrak{N}_*(\eta,\theta))\big\} & \text{for} & p\ge 0, \; 0\le q \le p+1\\
 \qquad 0 \hfill & \text{for} & 0\le p < q-1
\end{matrix}
\right.
\end{multline}
with $i_0, \ldots, i_p=0\;\text{or}\;1$.

Clearly, there is a natural morphism of bicomplexes
$C(\mathfrak{N}_*(\mu,\nu)): C(\mathfrak{N}_*(I,A,inc))\to {\mathbb
M}$, inducing the morphism of spectral sequences $f^1:E^1
\rightarrow \ol{E}^1$. To finish the proof it suffices to show that
$f^2_{p q}:E^2_{p q} \rra \ol{E}^2_{p q}$ is an isomorphism for any
$p,q\ge 0$.

We prove the remaining part of the assertion in two cases.

\noindent{\bf Case 1}. \emph{The homomorphism $f^2_{p q}:E^2_{p q}
\rra \ol{E}^2_{p q}$ is an isomorphism for any $p\ge 0$ and $q=0$}.

In effect, by Lemma \ref{lemexc1},  $E^2_{p0}=HH^{naive}_p(
\pi_0(\mathfrak{N}_*(I,A,inc)))$, while by (\ref{formula_3}) we have
$$
\CD \ol{E}^2_{p0}\cong H_p\Big( \Ker\{
C(\pi_0(\mathfrak{N}_*(S,B,\sigma)))@>{C(\pi_0(\mathfrak{N}_*(\eta,\theta)))}>>
C(\pi_0(\mathfrak{N}_*(T,C,\tau)) ) \}\Big).
\endCD
$$
By the assumption on the crossed module $(I,A,inc)$ and Corollary
\ref{cor_x} the algebra $\pi_0(\mathfrak{N}_*(I,A,inc))$ is
$H$-unital, then the result follows from Wodzicki's theorem
\cite{Wo} applied to the short exact sequence of algebras
(\ref{formula_4}).

\;

\noindent{\bf Case 2}. \emph{The homomorphism $f^2_{p q}:E^2_{p q}
\rra \ol{E}^2_{p q}$ is an isomorphism for any $p\ge 0$ and $q>0$}.

Using Lemma \ref{lemexc1} again, we have $E^1_{p q}=0$ and
consequently $E^2_{p q}=0$ for any $p\ge 0$ and $q>0$. Thus, we are
left to show that $\ol{E}^2_{p q}=0$, $p\ge 0$ and $q>0$. In effect,
by (\ref{formula_4}) and (\ref{formula_5}) we have the short exact
sequence of algebras
\begin{multline}\label{formula_8}
0 \lra \pi_0(\mathfrak{N}_*(I,A,inc))\lra \pi_1(\mathfrak{N}_*(S,B,\sigma))\rtimes \pi_0(\mathfrak{N}_*(S,B,\sigma))\\
\lra \pi_1(\mathfrak{N}_*(T,C,\tau))\rtimes \pi_0(\mathfrak{N}_*(T,C,\tau))\lra 0.
\end{multline}
Hence, the fact that $\pi_0(\mathfrak{N}_*(I,A,inc))$ is $H$-unital,
by Wodzicki's theorem \cite{Wo}, implies that we have the
quasi-isomorphism
\begin{equation}\label{formula_10}
\begin{matrix}
C\big(\pi_0(\mathfrak{N}_*(I,A,inc))\big)\rra \Ker\big\{C\big(\pi_1(\mathfrak{N}_*(S,B,\sigma))\rtimes \pi_0(\mathfrak{N}_*(S,B,\sigma))\big)\rra \\
C\big(\pi_1(\mathfrak{N}_*(T,C,\tau))\rtimes \pi_0(\mathfrak{N}_*(T,C,\tau)) \big) \big\}.
\end{matrix}
\end{equation}

Moreover, we have an isomorphisms of vectorspaces

\begin{equation}\label{formula_6}
\begin{matrix}
C_p\big(\pi_1(\mathfrak{N}_*(S,B,\sigma))\rtimes \pi_0(\mathfrak{N}_*(S,B,\sigma))\big)\cong(\pi_0(\mathfrak{N}_*(S,B,\sigma)))^{\otimes \;p+1}\\ \oplus \uns{i_0 + \cdots + i_p =1 }{\ovs{p+1}{\bigoplus}}\big( \pi_{i_0}(\mathfrak{N}_*(S,B,\sigma))\otimes \cdots \otimes \pi_{i_p}(\mathfrak{N}_*(S,B,\sigma)) \big)
\end{matrix}
\end{equation}
and
\begin{equation}\label{formula_11}
\begin{matrix}
C_p\big(\pi_1(\mathfrak{N}_*(T,C,\tau))\rtimes \pi_0(\mathfrak{N}_*(T,C,\tau))\big)\cong(\pi_0(\mathfrak{N}_*(T,C,\tau)))^{\otimes \;p+1}\\ \oplus \uns{i_0 + \cdots + i_p =1 }{\ovs{p+1}{\bigoplus}}\big( \pi_{i_0}(\mathfrak{N}_*(T,C,\tau))\otimes \cdots \otimes \pi_{i_p}(\mathfrak{N}_*(T,C,\tau)) \big)
\end{matrix}
\end{equation}
with $i_0, \ldots, i_p=0\;\text{or}\;1$.

Let us define $\mathfrak{D}_p^q$ and $\ol{\mathfrak{D}}_p^q$ for any
$p\ge 0$ and $q> 0$ by the formulas
$$
\mathfrak{D}_p^q=\left\lbrace
\begin{matrix}
\uns{i_0 + \cdots + i_p = q }{\bigoplus}\big( \pi_{i_0}(\mathfrak{N}_*(S,B,\sigma))\otimes \cdots \otimes \pi_{i_p}(\mathfrak{N}_*(S,B,\sigma)) \big)& \text{for} & q \le p+1\\
 \qquad 0 \hfill & \text{for} & p+1 < q
\end{matrix}
\right.
$$
and
$$
\ol{\mathfrak{D}}_p^q=\left\lbrace
\begin{matrix}
\uns{i_0 + \cdots + i_p = q }{\bigoplus} \big(\pi_{i_0}(\mathfrak{N}_*(T,C,\tau))\otimes \cdots \otimes \pi_{i_p}(\mathfrak{N}_*(T,C,\tau)) \big)& \text{for} & q \le p+1\\
\qquad 0 \hfill & \text{for} & p+1 < q
\end{matrix}
\right.
$$
with $i_0, \ldots, i_p=0\;\text{or}\;1$.

It is easy to show that $\mathfrak{D}^q$ is a sub-complex of $C\big(\pi_1(\mathfrak{N}_*(S,B,\sigma))\rtimes \pi_0(\mathfrak{N}_*(S,B,\sigma))\big)$ through the isomorphism (\ref{formula_6}) and along the Hochschild differential, which is left to the reader (hint: the induced multiplication in $\pi_1(\mathfrak{N}_*(S,B,\sigma))$ vanishes). Similarly, $\ol{\mathfrak{D}}^q$ is a sub-complex of $C\big(\pi_1(\mathfrak{N}_*(T,C,\tau))\rtimes \pi_0(\mathfrak{N}_*(T,C,\tau))\big)$. Moreover, by (\ref{formula_6}) and (\ref{formula_11}), there are decompositions in direct summands of complexes
\begin{align*}
C\big(\pi_1(\mathfrak{N}_*(S,B,\sigma))\rtimes \pi_0(\mathfrak{N}_*(S,B,\sigma))\big)\cong C(\pi_0(\mathfrak{N}_*(S,B,\sigma)))\oplus \; \uns{q\ge 1}{\bigoplus}\; \mathfrak{D}^q
\end{align*}
and
\begin{align*}
C\big(\pi_1(\mathfrak{N}_*(T,C,\tau))\rtimes \pi_0(\mathfrak{N}_*(T,C,\tau))\big)\cong C(\pi_0(\mathfrak{N}_*(T,C,\tau)))\oplus \; \uns{q\ge 1}{\bigoplus}\; \ol{\mathfrak{D}}^q .
\end{align*}
Hence
\begin{multline*}
\Ker\big\{C\big(\pi_1(\mathfrak{N}_*(S,B,\sigma))\rtimes \pi_0(\mathfrak{N}_*(S,B,\sigma))\big)\to C\big(\pi_1(\mathfrak{N}_*(T,C,\tau))\rtimes \pi_0(\mathfrak{N}_*(T,C,\tau)) \big) \big\}\\
\cong \Ker\big\{C(\pi_0(\mathfrak{N}_*(S,B,\sigma)))\xrightarrow{\pi_0(\mathfrak{N}_*(\eta,\theta))} C(\pi_0(\mathfrak{N}_*(T,C,\tau)))\big\}\oplus \uns{q\ge 1}{\bigoplus}\Ker\big\{ \mathfrak{D}^q \to \ol{\mathfrak{D}}^q \big\}.
\end{multline*}
Now the quasi-isomorphism (\ref{formula_10}) implies that the complex $\Ker \big( \mathfrak{D}^q \rra \ol{\mathfrak{D}}^q \big)$ is acyclic for any $q> 0$, which according to (\ref{formula_3}) means that $\ol{E}^2_{pq}=0$ for any $p\ge 0$ and $q>0$ .
\end{proof}

\

\section{Cotriple Cyclic Homology of Crossed Modules}\label{der}

In this section we assume that $\mathbf k$ is a field of characteristic zero.

\subsection{Adjunction}
We begin by constructing an adjoint pair of functors $\xymatrix@C=20pt{\Alg \ar@<0.4ex>[r]^-{F}&
\ar@<0.4ex>[l]^-{U} \XAlg}$.

Assume that the functor $U:\XAlg\to \Alg$ assigns to any crossed
module $(M,R,\mu)$ the direct product $M\times R$. Now, define the
functor $F:\Alg \rra \XAlg$ as follows: for any algebra $A$, let
$F(A)$ denote the inclusion crossed module of algebras
$(\overline{A},A\ast A, inc)$, where $A\ast A$ is the coproduct of
the algebra $A$ with itself, with inclusions $u_1,u_2:A\to A\ast A$,
and $\overline{A}$ is the kernel of the retraction $p_2:A\ast A\to
A$ determined by the conditions $p_2u_2=1_{A}$, $p_2u_1=0$.

\begin{proposition}\label{2.1}
The functor $F$ is left adjoint to the functor $U$.
\end{proposition}
\begin{proof} We state that, given an algebra $A$, the homomorphism
$$
(u_1,u_2): A\to \overline{A}\times(A\ast A)=UF(A)
$$
is a universal arrow from $A$ to the functor $U$. Indeed, let
$(S,B,\sg)$ be a crossed module and $f_{S}:A\to S$, $f_{B}:A\to B$
defining homomorphisms of any homomorphism $(f_S,f_B):A\to S\times
B=U(S,B,\sg)$. Then there is a commutative diagram with split short
exact sequences of algebras

$$
\xymatrix{
  \overline{A} \ar[r]^{inc \ \ } \ar[d]_{\al} & A\ast A \ar[d]_{\gamma} \ar@<0.6mm>[r]^{ \ \ p_2} & A \ar[d]^{f_B \;\;,}
  \ar@/_1.3pc/[l]_{u_2}  \\
  S \ar[r]^{i\ \ } & S\rtimes B \ar@<0.6mm>[r]^{\ \ pr } & B  \ar@/^1.3pc/[l]_{j}
  }
$$
where $\gm$ is defined by $\gm u_1=i f_S$ and $\gm u_2=j f_B$, and
$\al$ is the restriction of $\gm$. Let $\be:A\ast A \to B$ be the
unique homomorphism satisfying $\be u_1=\sg f_S$ and $\be u_2=f_B$.
Easy calculations show that $(\al,\be):(\overline{A},A\ast A,
inc)\to (S,B,\sg)$ is a morphism of crossed modules of algebras,
clearly the unique one such that
$(\al\times\be)(u_1,u_2)=(f_S,f_B)$.
\end{proof}

We denote by $W:\Alg\to \Vec$ the usual forgetful functor and by
$T:\Vec\to \Alg$ its left adjoint functor, carrying any vectorspace
$V$ to the free algebra on it. Composing these two adjunctions,
$$
\xymatrix@C=20pt{\Vec \ar@<0.4ex>[r]^-{T}&
\ar@<0.4ex>[l]^-{W} \Alg \ar@<0.4ex>[r]^-{F}&
\ar@<0.4ex>[l]^-{U} \XAlg} \;,
$$
we deduce the following.

\begin{proposition}\label{2.3}
The functor ${\mathbb F}=F\circ T:\Vec\to \XAlg$, $V\mapsto
({\overline{T(V)}},T(V)\ast T(V), inc)$, is left adjoint to the
functor ${\mathbb U}=W \circ U: \XAlg\to \Vec$, $(R,A,\rho)\mapsto
R\times A$.
\end{proposition}

\subsection{Construction and elementary properties}\label{sub_mose}
It is known due to \cite{DIL, IL} that the cyclic homology of
algebras is described as the non-abelian derived functors of the
additive abelianisation functor $Ab^{add}:\Alg\to\Vec$,
$Ab^{add}(A)=A/[A,A]$. To generalize the cyclic homology theory to
the category $\XAlg$ in terms of non-abelian derived functors we
need to extend the additive abelianisation functor to this category.
By reason of that we look the functor $Ab^{add}$ as a factorisation
through the category of Lie algebras $\Lie$. Explicitly, there is an
equality $Ab^{add}=Ab \circ {\mathfrak L}$, where ${\mathfrak
L}:\Alg\to \Lie$ is the classical Liesation functor and $Ab:\Lie\to
{\bf Vec}$ is the abelianisation functor of Lie algebras.

Let us construct the natural extension of the functors ${\mathfrak
L}$ and $Ab$ to the category $\XAlg$. First recall from \cite{KL}
that a crossed module of Lie algebras $(M,G,\mu)$ is a Lie
homomorphism $\mu:M\to G$ together with a bilinear map $G\times M\to
M$, $(g,m)\mapsto {}^gm$ satisfying
\par \begin{equation*}
{}^{[g,g']}m= {}^g{({}^{g'}m)}-{}^{g'}{({}^gm)},\qquad {}^g{[m,m']}=
[{}^gm,m']+[m,{}^gm'],
\end{equation*}
such that the following conditions hold:
\begin{align*}
(i)\;\mu({}^gm)=[g,\mu(m)],\quad (ii)\; {}^{\mu(m)}m'=[m,m'] \quad
\text{for all}\quad m,m'\in M,\; g\in G.
\end{align*}
Denote by $\XLie$ the category of crossed modules of Lie algebras.
Then there is a naturally defined functor ${\mathcal X}{\mathfrak
L}:\XAlg\to \XLie$ carrying a crossed module of algebras $\rho:R\to
A$ to the crossed module of Lie algebras ${\mathfrak
L}(\mu):{\mathfrak L}(R)\to {\mathfrak L}(A)$ with ${}^ar=ar-ra$.
Moreover, it is known that an abelian group object in the category
$\XLie$ is just a linear map of vectorspaces and their category is
denoted by $\XVec$. Now the abelianisation of crossed modules of Lie
algebras ${\mathcal{X}Ab}:\XLie\to \XVec$ is left adjoint to the
natural embedding functor $\XVec\subset \XLie$, which is explicitly
given by ${\mathcal{X}Ab}(M,G,\mu)=(M/[G,M], Ab(G), \wt\mu)$ where
$[G,M]$ denotes the subvectorspaces of the Lie algebra $M$ generated
by the elements ${}^gm$ for $g\in G$, $m\in M$, and $\wt\mu$ is the
Lie homomorphism induced by $\mu$.

From the aforementioned discussion we arrive to the definition of
the additive abelianisation functor of crossed modules of algebras
${\mathcal{X}Ab^{add}}:\XAlg\to \XVec$ as
${\mathcal{X}Ab^{add}}={\mathcal{X}Ab}\circ {\mathcal X}{\mathfrak
L}$.

Now we are ready to construct the \emph{cotriple cyclic homology of
crossed modules of algebras}. We assume the reader is familiar with
cotriples and projective classes. See, for example, \cite{bb2} and
\cite[Chapter 2]{InH1} for the background. The adjoint pair of
functors $\xymatrix@C=20pt{ \Vec \ar@<0.4ex>[r]^-{\mathbb F}&
\ar@<0.4ex>[l]^-{\mathbb{U}} \XAlg},$ constructed in the previous
subsection, induces a cotriple
$\mathcal{F}=(\mathcal{F},\delta,\tau)$ in $\XAlg$ by the obvious
way: $\mathcal{F}=\mathbb{F}\mathbb{U}:\XAlg\to \XAlg$,
$\tau:\mathcal{F}\to 1_{\XAlg}$ is the counit and $\delta=\mathbb{F}
u \mathbb{U}:\mathcal{F}\to \mathcal{F}^2$, where $u:1_{\Vec}\to
\mathbb{U}\mathbb{F}$ is the unit of the adjunction. Let $\mathcal
P$ denote the projective class induced by the cotriple $\mathcal F$:
$(R,A,\rho)\in {\mathcal P}$ iff there exists a morphism
$\vthe:(R,A,\rho)\to {\mathcal F}(R,A,\rho)$ such that
${\tau}_{(R,A,\rho)}\vthe =1_{(R,A,\rho)}$.

Given any crossed module $(R,A,\rho)$, there is an augmented
simplicial object $\mathcal{F}_*(R,A,\rho)\to (R,A,\rho)$ in the
category $\XAlg$, where
\begin{align*}
& {\mathcal F_n}(R,A,\rho)={\mathcal F^{n+1}}(R,A,\rho)={\mathcal F}\big({\mathcal F^n}(R,A,\rho)\big),\\
& d^n_i={\mathcal F^i}(\tau_{{\mathcal F^{n-i}}}),\quad s^n_i={\mathcal F^i}(\de_{{\mathcal F^{n-i}}}),\quad
0\le i\le n,
\end{align*} and which is called the {\em ${\mathcal F}$-cotriple
resolution} of $(R,A,\rho)$. Applying the functor
${\mathcal{X}Ab^{add}}$ dimension-wise to $\mathcal{F}_*(R,A,\rho)$
we obtain the simplicial object
$\mathcal{X}Ab^{add}\mathcal{F}_*(R,A,\rho)$ in the category
$\XVec$.

\begin{definition}\label{def4.1.4}
The $n$-th cotriple cyclic homology of a crossed module of algebras
$(R,A,\rho)$ is defined by
$$
{\mathcal HC}_n(R,A,\rho)=
H_n\big({\mathcal{X}Ab^{add}}(\mathcal{F}_*(R,A,\rho))\big),\quad
n\ge 0.
$$
\end{definition}

It is clear that ${\mathcal HC}_n$, $n\ge 0$ is a functor from
$\XAlg$ to $\XVec$. Moreover, for any $(R,A,\rho)\in \XAlg$,
$$
{\mathcal HC}_0(R,A,\rho)\cong
{\mathcal{X}Ab^{add}}(R,A,\rho)=\big(R/[A,R],
A/[A,A],\ol{\rho}\big),
$$
where $\ol{\rho}$ is the linear map induced by $\rho$.

For further investigation of the cotriple cyclic homology of crossed
modules of algebras we need some non-standard simplicial resolutions
in the sense of Barr-Beck \cite{bb2}.

\begin{proposition}\label{rem4.1.5}
Let $\Big((R_*,A_*,\rho_*),(d^0_0,d^0_0), (R,A,\rho)\Big)$ be an
augmented simplicial crossed module of algebras. Suppose the
following conditions hold:
\begin{enumerate}
\item[(i)] the crossed module $(R_n,A_n,\rho_n)$, $n\ge 0$, belongs to the projective class $\mathcal P$;
\item[(ii)] the augmented simplicial algebras $(R_*,d^0_0,R)$ and $(A_*,d^0_0,A)$ are aspherical.
\end{enumerate}
Then the simplicial crossed modules of algebras $(R_*,A_*,\rho_*)$
and $\mathcal{F}_*(R,A,\rho)$ are homotopically equivalent.
\end{proposition}
\begin{proof}
Straightforward from \cite[5.3]{bb2}.
\end{proof}

Now we describe several connections between the cyclic homology of
algebras and cotriple cyclic homology of crossed modules. There are
two ways of regarding an algebra $A$ as a crossed module, via the
trivial map $0:0\to A$ and via the identity map $1_A:A\to A$ with
action of $A$ on itself given by multiplication. Respectively there
are full embeddings
$$
i,\; \epsilon : \Alg\to \XAlg
$$
defined by $i A=(0,A,0)$ and $\epsilon A=(A,A,1_A)$. The functor $i$
has a left adjoint $\tau:\XAlg\to\Alg$, $\tau(R,A,\rho)=\Coker \rho$
and also a right adjoint $\kappa : \XAlg \to \Alg$,
$\kappa(R,A,\rho)=A$. On the other hand, the functor $\kappa$ and
the functor $\xi:\XAlg\to\Alg$, $\xi(R,A,\rho)=R$ are left and right
adjoint to the functor $\epsilon$, respectively.

Regarding $\XVec$ as a subcategory of $\XAlg$, the cotriple cyclic
homology $\mathcal{HC}_n(R,A,\rho)$, $n\ge 0$, could be presented as
a linear map  $\xi \mathcal{HC}_n(R,A,\rho)\to \kappa
\mathcal{HC}_n(R,A,\rho)$.

\begin{proposition}
\begin{enumerate}
\item[(i)] For any crossed module $(R,A,\rho)$ and $n\geq 0$,
$$
\kappa \mathcal{HC}_n(R,A,\rho)\cong HC_n(A).
$$
\item[(ii)] For any algebra $A$ and $n\geq 0$,
$$
\mathcal{HC}_n(i A)\cong i HC_n(A) \quad {\text{and}} \quad
\mathcal{HC}_n(\epsilon A)\cong \epsilon HC_n(A)\;.
$$
\end{enumerate}
\end{proposition}
\begin{proof}
\noindent (i) Given a crossed module $(R,A,\rho)$, by Proposition
\ref{rem4.1.5} the simplicial algebra $\kappa
\mathcal{F}_*(R,A,\rho)\to A$ is a free simplicial resolution of
$A$. Therefore
$$
\kappa \mathcal{HC}_n(R,A,\rho)=\kappa H_n({\mathcal{X}Ab^{add}}(\mathcal{F}_*(R,A,\rho)))=H_n(Ab^{add}(\kappa\mathcal{F}_*(R,A,\rho))).
$$
The assertion follows from \cite[Theorem 1.1]{DIL}.

\noindent (ii) Let $(A_*,d^0_0,A)$ be a free simplicial resolution
of an algebra $A$. It is routine to check that $i A_n$ and $\ep A_n$
belong to the projective class $\mathcal P$ for any $n\ge 0$. Then
by Proposition \ref{rem4.1.5} we have
$$
\mathcal{HC}_n(i A)\cong H_n({\mathcal{X}Ab^{add}}(i A_*))=
H_n(i{Ab^{add}}(A_*))= i H_n({Ab^{add}}(A_*))\cong i HC_n(A)
$$
and
$$
\mathcal{HC}_n(\epsilon A)\cong H_n({\mathcal{X}Ab^{add}}(\epsilon
A_*))= H_n(\epsilon{Ab^{add}}(A_*))= \epsilon
H_n({Ab^{add}}(A_*))\cong \epsilon HC_n(A).
$$
\end{proof}

Finally in this subsection we calculate the cotriple cyclic homology
of an inclusion crossed module of algebras.

\begin{proposition}\label{relat}
Let $(I, A, inc)$ be an inclusion crossed module of algebras and
$n\ge 0$. Then there is an isomorphism
$$
\xi \mathcal{HC}_n(I, A, inc)\cong HC_n(A,I),
$$
where $HC_n(A,I)$ denotes the $n$-th relative cyclic homology.
\end{proposition}
\begin{proof} Let $\mathcal{F}_*=\mathcal{F}_*(I, A, inc)\to (I, A, inc)$ be the cotriple resolution of $(I,A,inc)$. Note that both $\kappa\mathcal{F}_n$ and $\kappa\mathcal{F}_n/\xi\mathcal{F}_n$ are free algebras for each $n\geq 0$. Moreover, $\xi\mathcal{F}_*\to I$ and $\kappa\mathcal{F}_*\to A$ are aspherical augmented simplicial algebras and since $(I,A,inc)$ is an inclusion crossed module, the augmented simplicial algebra $\kappa\mathcal{F}_*/\xi\mathcal{F}_*\to A/I$ is aspherical as well.
We have the commutative diagram of complexes
\begin{equation}
\xymatrix {
CC(A) \ar[r] &  CC(A/I)\\
\Tot( CC(\kappa\mathcal{F}_*)) \ar[d]\ar[r]\ar[u] & \Tot(CC(\kappa\mathcal{F}_*/\xi\mathcal{F}_*)) \ar[d]\ar[u]\\
Ab^{add}(\kappa\mathcal{F}_*) \ar[r] &  Ab^{add}({\kappa\mathcal{F}_*/\xi\mathcal{F}_*})\;\;.
}
\end{equation}
Clearly, by Proposition \ref{prop_1}, the both vertical morphisms in the upper quadrant are quasi-isomorphisms. Furthermore, by \cite{DIL}, the both vertical morphisms in the lower quadrant are also quasi-isomorphisms. Consequently, we have an isomorphism
$$
HC_n(A,I)\cong H_n (\Ker
\{Ab^{add}(\kappa\mathcal{F}_*)\rightarrow
Ab^{add}(\kappa\mathcal{F}_*/\xi\mathcal{F}_*)\}
),\quad n\ge 0.
$$
Now from the five-term exact cyclic homology sequence of \cite{Qu2}
and the fact that $HC_1(\kappa\mathcal{F}_n/\xi\mathcal{F}_n)=0$,
$n\ge 0$ \cite{Lo}, we deduce
$$
\Ker
\{Ab^{add}(\kappa\mathcal{F}_*)\rightarrow
Ab^{add}(\kappa\mathcal{F}_*/\xi\mathcal{F}_*)\}
\cong \xi\mathcal{F}_*/[\kappa\mathcal{F}_*, \xi\mathcal{F}_*].
$$
But, by definition
$$
\xi \mathcal{HC}_n(I,A,inc)=\xi
H_n({\mathcal{X}Ab^{add}}\mathcal{F}_*)=
H_n(\xi\mathcal{F}_*/[\kappa\mathcal{F}_*, \xi\mathcal{F}_*]).
$$
This completes the proof.
\end{proof}

\begin{corollary}
Let $(I, A, inc)$ be an inclusion crossed module of algebras. and
$n\ge 1$. Then there is a long exact homology sequence
\begin{multline}\label{bolo}
\cdots \lra\xi\mathcal{HC}_n(I, A, inc)\lra HC_n(A)\lra
HC_n(A/I)\lra \cdots \lra HC_1(A/I)\lra I\big/[A,I]\\\lra
A\big/[A,A]\lra A\big/\big(I+[A,A]\big)\lra 0.
\end{multline}
\end{corollary}
\begin{proof}
Straightforward from Proposition \ref{relat}.
\end{proof}

\

\section{Cyclic homology vs. cotriple cyclic homology}

In this section we compare two above-discussed cyclic homology
theories of crossed modules of algebras in terms of long exact
homology sequence. Namely, the aim of this section is to prove the
following.

\begin{theorem}\label{conection}
Let $\mathbf k$ be a field of characteristic zero and $(R, A, \rho)$ a crossed module of algebras. Then there are natural
exact sequences
\begin{multline}\label{mimdevr_1}
\cdots \lra HC_{n+1}(R,A,\rho) \lra \xi {\mathcal HC}_n(R,A,\rho) \lra HC_n(A) \lra HC_n(R,A,\rho)\\
\lra \cdots \lra \xi {\mathcal HC}_1(R,A,\rho) \lra HC_1(A) \lra HC_1(R,A,\rho) \lra R/[A,R] \\ \ovs{\rho}{\lra} A/[A,A] \lra A\big/\big(\Im\rho+[A,A]\big)\lra 0
\end{multline}
and
\begin{multline}\label{mimdevr_2}
\cdots \lra \xi {\mathcal HC}_{n-1}(R,A,\rho) \lra H_{n+1}(\be(R,A,\rho)) \lra \xi {\mathcal HC}_n(R,A,\rho) \\
\lra \xi {\mathcal HC}_{n-2}(R,A,\rho) \lra \cdots \lra \xi {\mathcal HC}_1(R,A,\rho) \lra H_3(\be(R,A,\rho)) \lra \xi {\mathcal HC}_2(R,A,\rho)\\ \lra \xi {\mathcal HC}_0(R,A,\rho) \lra H_2(\be(R,A,\rho)) \lra \xi {\mathcal HC}_1(R,A,\rho) \lra 0.
\end{multline}
Moreover, there are isomorphisms
$$
H_1(\be(R,A,\rho))\cong \xi {\mathcal HC}_0(R,A,\rho)\cong R/[A,R],
$$
where the complex $\be$ is defined immediately below.
\end{theorem}

Note that the sequence (\ref{mimdevr_1}) is a natural generalisation
of the relative cyclic homology exact sequence (\ref{bolo}).

\subsection{The complexes beta and gamma} Given a crossed module $(R, A, \rho)$ of algebras, we have a natural morphism
of crossed modules
$$
(0,1_A):(0,A,0)\lra (R, A,\rho)\;.
$$
It is easy to see that there are injective maps of bicomplexes
\begin{align*}
CC^{\{2\}}(\mathfrak{N}_*(0,1_A)):CC^{\{2\}}(\mathfrak{N}_*(0,A,0))\to CC^{\{2\}}(\mathfrak{N}_*(R,A,\rho))
\end{align*}
and
\begin{align*}
CC(\mathfrak{N}_*(0,1_A)):CC(\mathfrak{N}_*(0,A,0))\to CC(\mathfrak{N}_*(R,A,\rho)),
\end{align*}
which yield the respective injective maps of complexes
$$
i_{(R,A,\rho)}:CC^{\{2\}}(0,A,0)\to CC^{\{2\}}(R,A,\rho)\quad \text{and} \quad j_{(R,A,\rho)}:CC(0,A,0)\to CC(R,A,\rho).
$$
Define the complex beta, $\be(R,A,\rho)$, and gamma, $\gm(R,A,\rho)$, from the following commutative diagram of complexes of vectorspaces with exact rows:
\begin{equation}\label{import}
\xymatrix {
0 \ar[r] & CC^{\{2\}}(0,A,0) \ar[d]\ar[r]^{{i_{(R,A,\rho)}}} & CC^{\{2\}}(R,A,\rho) \ar[d]\ar[r] & \be (R,A,\rho) \ar[d]\ar[r] & 0\\
0 \ar[r] & CC(0,A,0) \ar[r]^{{j_{(R,A,\rho)}}} & CC(R,A,\rho) \ar[r] & \gm (R,A,\rho) \ar[r] & 0
}
\end{equation}

The next two propositions calculate the low dimensional homology of $\be$ and $\gm$ complexes.

\begin{proposition}\label{prop_2}
Let $(R,A,\rho)$ be a crossed module of algebras. Then we have
$$
H_0\be(R,A,\rho) = H_0 \gm (R,A,\rho)=0
$$
and
$$
H_1\be(R,A,\rho) = H_1 \gm (R,A,\rho)\cong R/[A, R].
$$
\end{proposition}
\begin{proof}
Given a simplicial algebra $A_*$, the last two rows of the bicomplexes $CC^{\{2\}}(A_*)$ and $CC(A_*)$ coincide, which has the form for $\mathfrak{N}_*(0,A,0)$
$$
\CD
  & & A^{\otimes 2}\oplus A @<{}<< A^{\otimes 2}\oplus A @<<< A^{\otimes 2}\oplus A @<{}<< \\
  & & @VVV @VVV @VVV \\
  & & A @<{}<< A @<<< A @<{}<< ,
\endCD
$$
while for $\mathfrak{N}_*(R,A,\rho)$

$$
\CD
  & & A^{\otimes 2}\oplus A @<{}<<(R\rtimes A)^{\otimes 2}\oplus (R\rtimes A) @<<<
  (R\rtimes R\rtimes A)^{\otimes 2}\oplus (R\rtimes R\rtimes A) @<{}<< \\
  & & @VVV @VVV @VVV \\
  & & A @<{}<< R\rtimes A @<<< R\rtimes R\rtimes A @<{}<< .
\endCD
$$
Hence, it is clear that we have
$$
H_0\be(R,A,\rho) = H_0 \gm (R,A,\rho)=0.
$$
Moreover, comparing the given rows of the bicomplexes, we simply deduce that
$$
H_1\be(R,A,\rho) = H_1 \gm (R,A,\rho) = \Coker\Big( (R\rtimes A)^{\otimes
2}\big/A^{\otimes 2} \lra  (R\rtimes A)\big/A \Big)\;,
$$
where the arrow is defined as follows:
$$
(r,a)\otimes (r',a')\mapsto (r,a)(r',a')-(r',a')(r,a) = ([r,r']+[r,a']+[a,r'],[a,a']).
$$
This implies the second isomorphism of the assertion.
\end{proof}

Given an algebra $A$, by Corollary \ref{cor_x} the following pairs of complexes $CC^{\{2\}}(0,A,0)$, $CC^{\{2\}}(A)$ and $CC(0,A,0)$,
$CC(A)$ are quasi-isomorphic. Then taking into account Proposition \ref{prop_2}, for any crossed module $(R,A,\rho)$ of algebras, the diagram (\ref{import}) induces the morphism of long exact homology sequences

\begin{equation}\label{diag_1}
\xymatrix {
 \ar[r] & HH_n(A) \ar[d]\ar[r] & HH_n(R,A,\rho) \ar[d]\ar[r] & H_n\be(R,A,\rho) \ar[d]\ar[r] & \cdots \ar[r] & H_2\be(R,A,\rho)\ar[d]\\
 \ar[r] & HC_n(A) \ar[r] & HC_n(R,A,\rho) \ar[r] & H_n\gm(R,A,\rho) \ar[r] & \cdots \ar[r] & H_2\gm(R,A,\rho)\\
\ar[r] & HH_1(A) \ar[d]\ar[r] & HH_1(R,A,\rho) \ar[d]\ar[r] & R/[A, R] \ar@{=}[d]\ar[r] & HH_0(A) \ar@{=}[d]\ar[r] & HH_0(R,A,\rho) \ar@{=}[d]\ar[r] &  0\\
\ar[r] & HC_1(A) \ar[r] & HC_1(R,A,\rho) \ar[r] & R/[A, R] \ar[r] & HC_0(A) \ar[r] & HC_0(R,A,\rho) \ar[r] &  0.
}
\end{equation}

\begin{proposition}\label{prop_axali}
Let $(R,A,\rho)\in XAlg$ belong to the projective class $\mathcal P$ (see Subsection \ref{sub_mose}). Then
\begin{enumerate}
\item[(i)] $H_n\be (R,A,\rho)=0$ for any $n>2$;
\item[(ii)] if in addition $\mathbf k$ is a field of characteristic zero, we have
\begin{align*}
H_n\gm (R,A,\rho)=0\quad \text{for any}\quad n>1
\end{align*}
and an isomorphism
\begin{align*}
H_2\be (R,A,\rho)\cong R/[A, R].
\end{align*}
\end{enumerate}
\end{proposition}
\begin{proof}
Without loss of generality we can assume that $(R,A,\rho)={\mathbb
F}(V)$ for some $V\in \Vec$. Hence $R\ovs{\rho}{\hookrightarrow} A$
is an inclusion crossed module with $A$ and $A/R$ being free
algebras. Then, using again Corollary \ref{cor_x}, the pairs of
complexes $CC^{\{2\}}(R,A,\rho)$, $CC^{\{2\}}(A/R)$ and
$CC(R,A,\rho)$, $CC(A/R)$ are quasi-isomorphic.

Now (i) follows directly from the top exact sequence of the diagram
(\ref{diag_1}) and the fact that, for a free algebra $F$, the
Hochschild homology $HH_n(F)$ vanishes for any $n\ge 2$.

If in addition $\mathbf k$ is a field of characteristic zero, then
the cyclic homology $HC_n(F)$ of a free algebra $F$ vanishes as well
for any $n\ge 1$ (see \cite[Proposition 5.4]{Lo}). Therefore the
diagram (\ref{diag_1}) implies that $H_n\gm (R,A,\rho)=0$ for $n>1$.
Besides, by the Connes' periodic exact sequence there is a natural
isomorphism
$$
HC_0(F)\ovs{\ovs{B}{\cong}}{\lra}HH_1(F).
$$
Then the five term exact sequence in cyclic homology from \cite{Qu2}
implies the commutative diagram with exact rows
$$
\xymatrix {
0 \ar[r] & R/[A, R] \ar[d]\ar[r] & HC_0(A) \ar[d]_{\cong}\ar[r] & HC_0(R,A,\rho)\ar[d]^{\cong}\; \cong HC_0(A/R) \ar[r] & 0\\
0 \ar[r] & H_2\be(R,A,\rho) \ar[r] & HH_1(A) \ar[r] & HH_1(R,A,\rho)\; \cong HH_1(A/R) \ar[r] & 0
}
$$
which completes the proof.
\end{proof}

\subsection{Proof of Theorem \ref{conection}}
Given a simplicial crossed module $(R_*,A_*,\rho_*)$ and a functor
$\varPhi:\XAlg\to {\mathfrak C}_{\ge 0}$, denote by
$\varPhi(R_*,A_*,\rho_*)$  the bicomplex of vectorspaces obtained by
applying the functor $\varPhi$ dimension-wise to the simplicial
crossed module $(R_*,A_*,\rho_*)$.

The following lemma will be needed.

\begin{lemma}\label{lem_5.4}
Let $\Big((R_*,A_*,\rho_*),(d^0_0,d^0_0), (R,A,\rho)\Big)$ be an
augmented simplicial crossed module of algebras. Suppose $(R_*,
d^0_0, R)$ and $(A_*, d^0_0, A)$ are aspherical augmented simplicial
algebras. Then the augmented simplicial vectorspaces
$$
(\be_n(R_*,A_*,\rho_*),\be_n(d^0_0,d^0_0), \be_n (R,A,\rho)) \quad
\text{and} \quad (\gm_n(R_*,A_*,\rho_*),\gm_n(d^0_0,d^0_0), \gm_n
(R,A,\rho))
$$
are acyclic for any $n\ge 0$.
\end{lemma}
\begin{proof}
We shall prove only the acyclicity of the second augmented
vectorspace. The proof for the first is similar.

In fact, using the fact that the semi-direct product of aspherical
simplicial algebra is aspherical as well, the augmented simplicial
algebra
$$
(E_q(R_*,A_*,\rho_*),E_q(d^0_0,d^0_0), E_q (R,A,\rho)),\quad q\ge 0,
$$
is aspherical. Now by Proposition \ref{prop_1} (i) we have that the augmented simplicial vectorspace
$$
(CC_p(E_q(R_*,A_*,\rho_*)),CC_p(E_q(d^0_0,d^0_0)), CC_p(E_q (R,A,\rho))),\quad p\ge 0,\; q\ge 0
$$
is acyclic. This clearly implies that the augmented simplicial vectorspaces
$$
(CC_n(R_*,A_*,\rho_*), CC_n(d^0_0,d^0_0), CC_n (R,A,\rho))
$$
and consequently
$$
(CC_n(0,A_*,0), CC_n(0,d^0_0), CC_n (0,A,0))
$$
are acyclic for any $n\ge 0$. Clearly, the short exact sequence of augmented simplicial vectorspaces

$$
\xymatrix {
0\ar[r] & CC_n(0,A_*,0)\; \ar[d]_{CC_n(0,d^0_0)}\ar[r] & CC_n(R_*,A_*,\rho_*) \ar[d]_{CC_n(d^0_0,d^0_0)}\ar[r] & \gm_n(R_*,A_*,\rho_*) \ar[d]^{\gm_n(d^0_0,d^0_0)}\ar[r]& 0\\
0\ar[r] & CC_n (0,A,0)\; \ar[r] & CC_n (R,A,\rho) \ar[r] & \gm_n (R,A,\rho)\ar[r]& 0,
}
$$
implies the result.
\end{proof}

Return to the proof of Theorem \ref{conection}.

Consider the bicomplex $\gm(R_*,A_*,\rho_*)$, where $\Big((R_*,A_*,\rho_*),(d^0_0,d^0_0), (R,A,\rho)\Big)$ is a simplicial resolution of $(R,A,\rho)$ in $\XAlg$ in the sense of Barr-Beck \cite{bb2} (see Proposition \ref{rem4.1.5}). For any fixed $q$ the homology of the complex $\gm_q(R_*,A_*,\rho_*)$, by Lemma \ref{lem_5.4}, is $H_p\big(\gm_q(R_*,A_*,\rho_*)\big)=0$ if $p>0$ and $H_0\big(\gm_q(R_*,A_*,\rho_*)\big)\cong \gm_q(R,A,\rho)$. Therefore $H_n\big(\Tot (\gm(R_*,A_*,\rho_*))\big)\cong H_n(\gm(R,A,\rho))$. On the other hand, there is a spectral sequence
\begin{align*}
E^1_{pq}=H_q\big(\gm(R_p,A_p,\rho_p)\big)\Longrightarrow H_{p+q}(\gm(R,A,\rho)).
\end{align*}
But by Proposition \ref{prop_2} and Proposition \ref{prop_axali} we have
$$
E^1_{pq}=\left\lbrace
\begin{matrix}
0\qquad \qquad \qquad & \text{for} & p\ge 0, \; q\neq 1\\
R_p\big/[A_p,R_p] & \text{for} & p\ge 0, \; q=1.
\end{matrix}
\right.
$$
Moreover, $E^2_{pq}\cong \xi {\mathcal HC}_p(R,A,\rho)$ for $q=1$,
$p\ge 0$. Then the degenerated spectral sequence $E^2_{pq}$ yields
the natural isomorphism
\begin{equation}\label{imp_iso}
H_{n+1}(\gm (M,R,\mu))\cong \xi {\mathcal HC}_n(M,R,\mu)\;,\quad n\ge 0\;.
\end{equation}
Thus (\ref{diag_1}) and (\ref{imp_iso}) imply the exact sequence
(\ref{mimdevr_1}). Furthermore, from (\ref{import}) one easily
deduces that for any crossed module $(R,A,\rho)$ there is a short
exact sequence of complexes
\begin{equation*}\label{mimdevr_4}
0\lra \be(R,A,\rho)\lra \gm(R,A,\rho)\lra \gm(R,A,\rho)[2]\lra 0,
\end{equation*}
where $\gm(R,A,\rho)[2]$ is the dimension shifted complex by $2$,
i.e. $\gm_n(R,A,\rho)[2]=\gm_{n-2}(R,A,\rho)$, $n\ge 0$. Now the
induced long exact homology sequence with the isomorphism
(\ref{imp_iso}) completes the rest part of the theorem.

\

\end{document}